\documentclass{amsart}
\usepackage[dvips]{color}

\usepackage{graphicx}
\usepackage{epsfig}
\usepackage{tikz}

\usepackage{latexsym}
\usepackage{amssymb}
\usepackage{amsmath}
\usepackage{amsthm}
\usepackage{amscd}
\theoremstyle{plain}

\newtheorem{thm}{Theorem}
\newtheorem*{thm*}{Theorem}
\newtheorem{lem}{Lemma}
\newtheorem{defin}{Definition}
\newtheorem{prop}{Proposition}

\theoremstyle{definition}
\newtheorem{rem}{Remark}

\newtheorem{ques}{Question}

\newcommand{\nc}{\newcommand}
\nc\bR{\mathbb{R}}
\nc\bZ{\mathbb{Z}}

\newcommand{\calH}{{\mathcal H}}

\newcommand{\R}{{\mathbb R}}

\newtheorem{cor}{Corollary}
\title[Circle averages and disjointness]{Circle averages and disjointness in typical translation surfaces on every Teichm\"uller disc}
\begin{document}

\author{Jon Chaika, Pascal Hubert}

%\address{}

%\email{}
\address{
University of Utah 
Department of Mathematics, 203
155 S 1400 E, Room 233
Salt Lake City, UT, 84112-0090
USA}
\address{I2M, Centre de Math\'ematiques et Informatique (CMI),
Universit\'e Aix-Marseille, 39 rue Joliot Curie, 13453 Marseille Cedex 13, France.}
\email{chaika@math.utah.edu}
\email{pascal.hubert@univ-amu.fr}
\maketitle
\begin{abstract}
We prove that on the typical translation  surface the flow in almost every pair of directions are not isomorphic to each other and are in fact disjoint. It was not known if there were any  translation  surfaces other than torus covers with this property. We provide an application to the convergence of `circle averages' for the flow (away from a sequence of radii of density 0) for such surfaces. Even the density of a sequence of 'circles' was only known in a few special examples. MSC classes: 	37A10, 37A25, 37A34, 37E35
\end{abstract}
The illumination problem is a classical one in billiard theory (see for instance \cite{LMW} and references therein). A light source is located at some point in the billiard table: we wonder which part of the table is eventually illuminated.
 This question has  recently been solved in full generality for rational polygons and translation surfaces by Leli\`evre, Monteil and Weiss \cite{LMW} using deep results of Eskin and Mirzakhani on moduli spaces of translation surfaces \cite{Eskin-Mirzakhani}. Here we tackle a related question. We want to understand how big light ``circles" distribute in translation surfaces. This question was orally addressed more than 10 years ago by Boshernitzan to the second named author. It was also asked in \cite[Section 0.1.5 Page 13]{Mon}. 
 
  Formally, let $(X,\omega)$ denote a compact translation surface (with distinguished vertical direction). Let $F_{2\pi\theta,\omega}^t$ denote the linear flow in direction $\theta$ at time $t$ on $(X,\omega)$  and $\lambda_\omega$ denote the (2-dimensional) area on $(X,\omega)$ normalized to have area 1. \footnote{In the sequel, we will abreviate $(X,\omega)$ by $\omega$ and $F_{2\pi\theta,\omega}^t$ by $F_{2\pi\theta}^t$ when no confusion is possible. $F_\omega$ is the vertical flow on $(X,\omega)$.}

\begin{defin}A translation surface is \emph{illuminated by circles} if $$\underset{t \to \infty}\lim \, \int_0^1 h(F_{2\pi\theta}^tp) d\theta=\int_X h d\lambda_\omega$$ for all  points  $p$ in $X$ and  all continuous functions $h$.
\end{defin}
%\textit{Something about the now solved illumination problem. Provides motivation...}

\begin{ques}Is the typical surface illuminated by circles? Is every surface illuminated by circles?
\end{ques}
It is easy to see that a flat torus is illuminated by circles since a piece of a large circle has small curvature and can be approximated by a segment. For a translation surface, this is an open problem. The main difference is that a big ``circle" on a translation surface of higher genus is a union of disjoint small arcs. The size of each arc decreases when the radius of the circle grows.

We prove a partial result that requires  definitions:

\begin{defin} Let $A \subset \mathbb{R}$. The \emph{density of $A$} is $\underset{N \to \infty}{\liminf}\, \frac{\lambda(A \cap [-N,N])}{2N}$. 
\end{defin}
\begin{defin} A surface is \emph{weakly illuminated by circles} if for each $p$ there exists a set   of density 1, $G_p\subset \mathbb{R}$ so that 
\begin{equation}\label{eq:illuminate}
\underset{t \in G_p}{\lim} \int_0^1 h(F^t_{2\pi \theta}(p))d\theta=\int hd\lambda_\omega
\end{equation}
 for all  continuous functions  $h$.
\end{defin}

\begin{thm}\label{thm:weakly seen}
Almost every surface is weakly illuminated by circles. \footnote{Almost is respect to any $SL_2(\mathbb{R})$ invariant probability measure in a moduli space of compact translation surfaces.}
 In fact, if $\omega$ is a translation surface then for almost every $A \in SL_2(\mathbb{R})$, $A\omega$ is weakly illuminated by circles.
\end{thm}
The weaker question of whether circles became dense was also open (and is resolved by the previous theorem). That is:
\begin{cor} Almost every surface \footnote{Almost is respect to any $SL_2(\mathbb{R})$ invariant probability measure in a moduli space of compact translation surfaces.}  has the property that for any  $\epsilon>0$ there is a $T$ so that $\cup_{\theta \in 2\pi} F^T_{\theta}(p)$ is $\epsilon$ dense (in the usual flat metric on the surface).
\end{cor}
This answers a question in \cite{Mon}.

We derive  Theorem \ref{thm:weakly seen} from
\begin{thm}\label{thm:typ disjoint} For almost every surface \footnote{ Almost is respect to any $SL_2(\mathbb{R})$ invariant probability measure in a moduli space of compact translation surfaces.} for any $k \in \mathbb{N}$ we have 
$$\lambda^k(\{(\theta_1,...,\theta_k):F_{\theta_1}\times ...\times F_{\theta_k} \text{ is uniquely ergodic} \})=1$$
where $\lambda$ is the (normalized) Lebesgue measure on the circle.
 Moreover, for every $\omega$ then for almost every $A\in SL_2(\mathbb{R})$ we have that 
 $$\lambda^k(\{(\theta_1,...,\theta_k):F_{\theta_1, A\omega}\times ...\times F_{\theta_k, A\omega} \text{ is uniquely ergodic} \})=1.$$
\end{thm}
\begin{cor} For almost every surface the flow in almost every direction is not isomorphic to the vertical flow.
\end{cor}
 Before this result it was not known whether for every surface, other than torus covers, there was a single isomorphism class (depending on the surface) so that the flow in almost every direction was in this isomorphism class. 
This is a strengthening of a result by Gadre and the first named author \cite{Disjoint flow} (which ruled out that there was one isomorphism class for almost every translation surface).

\subsection{Organization of the paper} 
The condition that a surface is weakly illuminated by circles is approachable from general ergodic theory. In section \ref{sect:reduction}, we prove that Theorem \ref{thm:typ disjoint} (for $k=2$) implies Theorem \ref{thm:weakly seen}. In section  \ref{sect:disjointness}, we provide an abstract disjointness criterion which is a refinement of the main result in \cite{Disjoint} . We apply this criterion to translation flows in section \ref{sect:application} using  a matrix decomposition.  Given two directional flows $F_{\theta_1}$ and $F_{\theta_2}$, the $SL_2(\mathbb{R})$ deformation allows  us  to match two sets of real numbers together: these two sets are defined in section \ref{sect:disjointness}  (Definitions \ref{def:part rig} and \ref{def:spread}).  One is defined in terms of $F_{\theta_1}$ and the other one in terms of  $F_{\theta_2}$. 
\\

\textbf{Acknowledgments:} We thank Sebastien Gou\"ezel for a helpful conversation (he found an important simplication of section \ref{sect:reduction}:  the proof of Proposition \ref{prop:seb}). We thank the anonymous referee for many corrections and helpful suggestions that improved the paper. In particular, the current proof of Proposition 2 is due to the referee. We thank Oberwolfach, where the project began and CIRM where it was completed. J. Chaika was supported in part by NSF grants DMS-1300550 and DMS-1452762, the Sloan foundation and a Warnock chair. J. Chaika thanks Giovanni Forni for bringing this question to his attention. P. Hubert is partially supported by Projet ANR blanc GeoDyM.

\section{Background}

We will freely use the language of translation surfaces and ergodic theory. 
Concerning the background on translation surfaces, see for instance the following surveys \cite{Forni-Matheus}, \cite{MT}, \cite{viana survey},  \cite{Zo}. To learn more about ergodic theory, especially about joinings, see \cite{Ru} and \cite{glasner}. 

%Maybe move Section 3.2 (spectral preliminaries) or at least the second lemma of it to this section.

\subsection{Translation surfaces}

A translation surface $X$ is a compact surface of genus $g$ endowed with a flat metric with trivial rotational holonomy
and conical singularities whose angles are multiples of $2\pi$. Alternatively, a translation surface $X$ is a
datum $(S,\omega)$, where $S$ is a compact Riemann surface of genus $g$ and $\omega$ is an holomorphic 1-form on $S$
with zeros of orders $k_1, \dots ,k_r$ at points $p_1, \dots ,p_r$.
The linear flow $F_\theta$ is well defined for every direction $\theta$. Kerckhoff, Masur, Smillie showed that $F_\theta$ is uniquely ergodic for almost every $\theta$ (\cite{KMS}).
A maximal subset of $X$ filled by parallel closed geodesics is called an (open) cylinder. 
\\

 For a translation surface $X$, the genus and the orders of zeroes satisfy the relation $k_1 + \cdots +k_r = 2g-2$.  For fixed integers $k_1, \dots ,k_r$ satisfying the last relation, let $\calH(k_1, \dots, k_r)$ denote   the corresponding stratum of the moduli space of translation surfaces, that is the set of translation surfaces whose associated 1-form $\omega$ has $r$ zeroes with orders $k_1, \dots ,k_r$. It is a complex orbifold with complex dimension 
 $2g + r -1$. Consider a translation surface $X = (S, \omega)$ in the stratum $\calH(k_1, \dots, k_r)$ and $A  \in SL(2,\R)$. A new translation surface $A · X$  is obtained by the linear action of $A$ in the translation charts. Therefore the group $SL(2, \R)$ acts on $\calH(k_1, \dots, k_r)$. The group $SL(2,\R)$ preserves the hypersurface $\calH^{1}(k_1, \dots, k_r)$ consisting of all $X \in \calH(k_1, \dots, k_r)$ with $Area(X) = 1$. The Teichm\"uller flow $g_t$ is the action of the diagonal subgroup, that is
$g_t = \begin{pmatrix}
 e^t & 0 \\
 0 & e^{-t}
 \end{pmatrix}$,
 the unstable and stable horocycle flows are 
 $h_s = \begin{pmatrix}
 1 & s \\
 0 & 1
 \end{pmatrix}$, $\hat{h}_s=\begin{pmatrix} 1&0\\s&1\end{pmatrix}$, 
 the circular flow is 
 $r_\theta = \begin{pmatrix}
 \cos(\theta) & \sin(\theta)\\
 -\sin(\theta) & \cos(\theta)
 \end{pmatrix}$.
 Our results hold for almost every surface with respect to any $SL(2, \mathbb{R})$ invariant measure on the space of translation surfaces and does not require the recent classification of these measures due to Eskin-Mirzakhani \cite{Eskin-Mirzakhani}.

\subsection{Spectral theory}

Let $\calH$ be a Hilbert space and $T$ be a bounded linear operator acting on $\calH$.
A complex number $\lambda$ belongs to the spectrum of $T$ if $T- \lambda Id$ is not invertible. 
%The point spectrum is the closure of the span of the set of eigenvalues of $T$. 
%It complement in the spectrum is the continuous spectrum. 
 We denote by
$\calH_{pp}$  the closure of the eigenfunctions and by $\calH_c$ its orthogonal complement in $\calH$. 
$\calH_c$ is the subset of $\calH$ with continuous spectrum.

The link with ergodic theory is the following:
Let $X$ be a polish space, $\phi_t$ a flow preserving a measure $\mu$ on $X$ and $\calH = L^2(X, \mu)$, the family of unitary operators $U_t$ is defined by
$U_t (f) = f \circ \phi_t$ for $t \in \R$.

By Stone's theorem, there exists a self-adjoint operator $T$ such that $U_t = e^{it T}$. We call $T$ the infinitesimal generator of the family of operators $U_t$. A complex number $z$ belongs to the spectrum of $U_t$ if and only if $z = e^{ita}$ where $a$ is a (real) number in the spectrum of $T$.  Moreover if $f$ is an eigenfunction of $T$ with eigenvalue $a$ then $f$ is an eigenfunction of $U_t$ with eigenvalue $e^{ita}$.\footnote{To simplify notations, we will write $exp(s)$ as  $e^{ i s}$.} Consequently, the spaces $\calH_{pp}$ and $\calH_c$ do not depend on $t$. 

We recall that ergodicity, the mixing property and the weak mixing property are spectral properties. For instance, the flow $\phi_t$ is weak mixing if and only if the operator $U_t$ has no non trivial eigenfunctions (the space $\calH_{pp}$ contains only constant functions). Moreover, there is a subset $A$ of density 1 in $\R$ so that along $A$, $U_t$ restricted to $\calH_c$ converges to 0 in the weak operator topology. That is, for each $f$ and $g$ in $\calH_c$
$$\underset{t \to \infty, t \in A}{\lim}<U_tf,g> =0.$$

In the sequel, we will need the following 
\begin{lem}\label{lem:product spectrum}Let $F_1$ and $F_2$ be two flows. Let $\mathcal{H}_{c_1}$, $\mathcal{H}_{c_2}$ be the subsets of $L^2$ with continuous spectrum for each. Let $A_1,A_2$ be sets of $\R$ where $U_{F_i}|_{H_{c_i}}$ converge to 0 in the weak operator topology. Let $\{exp(s\alpha_i)\}_{i=1}^{\infty}$ and $\{exp(s\beta_i)\}_{i=1}^{\infty}$ be the eigenvalues of $F_1^s$ and $F_2^s$ respectively. Then $A_1 \cap A_2$ is a set so that $F_1\times F_2 $ converges to 0 in the weak operator topology on the subset of $L^2$ with continuous spectrum. The eigenvalues of $(F_1\times F_2)^s$ have the form $exp(s(\alpha_i+\beta_j))$. 
\end{lem}
We leave this as an exercise to the reader. 

% behavior on H_c

\subsection{Joinings}

We  recall some standard material (see \cite[page 132]{glasner}).
Let $F_1^s: X \to X$ and $F_2^s: Y \to Y$ be two flows. Assume that $F_1$ preserves the probability measure $\mu$ and $F_2$ preserves the probability measure $\nu$. A joining $\lambda$ is an invariant measure by $F_1\times F_2$ on the space $X \times Y$ with marginals $\mu$ and $\nu$. More precisely, for every measurable sets $A$ in $X$ and $B$ in $Y$.
$$\lambda(B\times X) = \mu(B) \textrm{ and } \lambda(X\times A) = \nu(A).$$ We say $(F_1,\mu)$ and $(F_2,\nu)$ are \emph{disjoint} if $\mu \times \nu$ is their only joining.

 A continuous linear map $P:L^2((X,\mu)) \to L^2((Y,\nu))$ is called a \emph{Markov operator} if \begin{enumerate}
\item $P\geq 0$ and $P^*\geq 0$,
\item $P\chi_X=\chi_Y$ and $P^*\chi_Y=\chi_X$
\item $PU_{F_1}^s=U_{F_2}^sP$. 
\end{enumerate}
Let $\lambda$ be an $F_1 \times F_2$ preserved measure of $X \times Y$ with marginals $\mu$ and $\nu$. This defines a  Markov operator $\Phi:L^2(\nu) \to L^2(\mu)$  by  $\int_B\Phi(\chi_A)d\mu=\lambda(B \times A)$. Formally, this Markov operator is the conditional expectation associated to the disintegration of $\lambda$ over $\nu$. The set of Markov operators are in 1-1 correspondence with joinings. This identification respects the convex structure of preserved measures and so  the extreme points come from ergodic joinings. %$\Phi$ intertwines the dynamics: 
%$\Phi_1 \circ U_{F_1}^s=U_{F_2}^s \circ \Phi_1$ and similarly for $\Phi_2$.

%Let $\omega$ denote a flat surface (with distinguished vertical direction). Let $F_\omega$ denote the vertical flow on $\omega$ and $\lambda_\omega$ denote (2-dimensional) area on $\omega$ normalized to have area 1. Let $\lambda$ denote Lebesgue measure on $S^1$.

\section{Theorem \ref{thm:typ disjoint} implies Theorem \ref{thm:weakly seen}}\label{sect:reduction} 
%\ann{Notation: In later sections we use $\lambda_\omega$ for the measure on the surface. In this section we use $\lambda^2$. Which to use? Are we using $\lambda$ too much and maybe should use $\mu_\omega$? Or should we just use $\lambda^2$ because it is a nice 2-manifold.}\ann{Should we name the set: $\{\bar{\theta}: F_{\bar{\theta}}^s(\bar{p})\in (\{\bar{v} \in (X,\omega)^M:|\frac 1 M\sum_{i=1}^M\chi_R(v_i)-\lambda^2(R)|>\epsilon\}\}$ or the set $\{\bar{v} \in (X,\omega)^M:|\frac 1 M\sum_{i=1}^M\chi_R(v_i)-\lambda^2(R)|>\epsilon\}$?}

This section describes the reduction of Theorem \ref{thm:weakly seen} to Theorem \ref{thm:typ disjoint} for $k=2$.

\begin{prop}\label{prop:seb} Let $\omega$ be a  translation surface. If $F_{r_\psi\omega}\times F_{r_\phi \omega}$ is uniquely ergodic for almost every $(\psi, \phi)$ then for any $f \in C(\omega)$ 
there exists a set of density 1, $A$, so that 
$$\underset{t \in A}{\lim}\int_0^{2\pi}f(F_{r_\theta\omega}^tp)d\theta=\int f d\lambda_\omega.$$
\end{prop}
\begin{proof}It suffices to show that $$\underset{T \to \infty}{\lim}\, \frac 1 T \int_0^T(|\int_0^{2\pi}f(F_{r_\theta\omega}^tp)d\theta|)^2 dt=0$$ for any $f \in C(\omega)$ with integral 0.
By Fubini's Theorem and the fact that $f$ is bounded (so we may interchange the limit and integral) this is
$$\int_0^{2\pi}\int_0^{2\pi}\underset{T \to \infty}{\lim} \, \frac 1 T \int_0^Tf(F_{r_\theta\omega}^tp)\bar{f}(F_{r_\phi\omega}^tp )dtd\theta d\phi.$$
 By our assumption that $F_{r_\theta\omega}\times F_{r_\phi \omega}$ is uniquely ergodic, this is 0.\footnote{$G(p,q)=f(p)\bar{f}(q)$ is continuous and has integral 0.}
\end{proof}
\begin{proof}[Proof of Theorem \ref{thm:weakly seen} assuming Theorem \ref{thm:typ disjoint}] This is a standard argument.
By Theorem 2, we have the assumption of Proposition \ref{prop:seb}. So for each $p$, $f\in C(X)$ we have that $$\underset{T \to \infty}{\lim} \frac 1 T \int _0^T|\frac 1 {2\pi}\int _0^{2\pi}f(F^t_{r_\theta \omega}p)d\theta-\int f d\lambda|=0.$$ Choose a countable subset $f_1,...$ of $C(X)$ that is dense in supremum norm. It suffices to show that there exists a set $\hat{A}$ of density 1 so that 
$$\underset{t \in \hat{A}}{\lim} \frac 1 {2\pi}\int_0^{2\pi} f_j(F^t_{r_\theta}p)d\theta=\int f_jd\lambda_\omega$$ for all $j$.
 Let $$a_n(t)=| \frac 1 {2\pi}\int_{0}^{2\pi}f_n(F^t_{r_{\theta}\omega}p) d\theta- \int f_n d\lambda_\omega\vert.$$
Clearly, $\frac 1 T \int_0^T \sum_{j=1}^{\infty} \frac{a_j(t)}{2^j\|f_j\|_{\sup}} dt \to  0.$ So there exists a sequence of density 1, $\hat{A}$, so that for all $n$ we have $a_n(t) \to 0$ as $t \to \infty$ in $\hat{A}$.  
 \end{proof}
\section{Disjointness criterion} \label{sect:disjointness}
The main result of this section is our disjointness criterion, Proposition \ref{prop:disjoint criterion}, which requires some preliminaries.
\begin{defin}\label{def:part rig} Let $(X,F^s,\mu)$ be an ergodic measure preserving flow on a metric space $(X,d)$. We say \{$t_1\}_{i=1}^\infty$ is a $c$-partial rigidity sequence for $F^s$ if there exists sets $S_1, \dots$ so that
\begin{enumerate} 
\item $\mu(S_i)\geq c$
\item $\underset{i \to \infty}{\lim}\int_{S_i}  d(F^{t_i}x,x)d\mu(x)=0$
\item for each $s$ we have $\underset{i \to \infty}{\lim} \mu(F^s S_i \Delta S_i)=0.$
\end{enumerate}
\end{defin}
\noindent
 Though we include condition (3) in the statement, it follows from the other 2 conditions.
\begin{lem}A system satisfying (1) and (2) satisfies (3).
\end{lem}
\begin{proof} Let $\delta>0$ and $s$ be given. It suffices to show there exists $S_i$ so that 
\begin{itemize}
\item $\mu(S_i)>c-9\delta$ for all but finitely many $i$, 
\item $\underset{i \to \infty}{\lim}\int_{S_i}  d(F^{t_i}x,x)d\mu(x)=0$ and  
\item $ \mu(F^s S_i \Delta S_i)<9\delta$ for all but finitely many $i$. 
\end{itemize}
Let $K$ be a compact set with $\mu(K)>1-\delta$ so that $F^s$ is (uniformly) continuous on $K$. 
Let $f:(0,\infty) \to (0,\infty)$ monotonic so that 
\begin{itemize}
\item $\underset{\epsilon \to 0}{\lim} \, f(\epsilon)=0$ and
\item if $p,q \in K$ and $d(p,q)<\epsilon$ then $d(F^sp,F^sq)<f(\epsilon)$. 
\end{itemize}
It suffices to show that for all large enough $i$ there exists an $\epsilon_i$, with $\underset{ i \to \infty}{\lim} \epsilon_i=0$, so that 
$$\mu(\{p:d(F^{t_i}p,p)<\epsilon_i\})>c-\delta \text{ and }\mu(\{p:\epsilon_i<d(F^{t_i}p,p)<f(\epsilon_i)\})<\delta.$$
Indeed let $S_i=\{p \in K \cap F^{-t_i}K:d(F^{t_i}p,p)<f(\epsilon_i)\}$ and observe that if $p\in S_i \setminus F^s(S_i)$ then $\epsilon_i<d(F^{t_i}p,p)<f(\epsilon_i).$ The existence of such $\epsilon_i$ is straightforward because for any map $G$, there exists at most $\delta^{-1}$ different $j$ so that 
$$\mu(\{p:f^j(x)<d(Gx,x)<f^{j+1}(x)\})>\delta.$$ Thus we choose $\epsilon_i '\to 0$ so that $f^{\delta^{-1}+2}(\epsilon_i') \to 0$ and $\mu(\{p:d(F^{t_i}p,p)<f^{\delta^{-1}+2}(\epsilon_i')\})>c-\delta$ and our $\epsilon_i$ will be $f^j(\epsilon_i')$ for some $j\leq \delta^{-1}+2$ so that $\mu(\{p:\epsilon_i<d(F^{t_i}p,p)<f(\epsilon_i)\})<\delta.$ \end{proof}

%Note that (3) follows from (2) via a straightforward application of Lusin's Theorem. 

  Also note that (3) and the ergodicity of $F^s$ imply  that $\underset{i \to \infty}{\lim} \, \frac 1 {\mu(S_i)}<\chi_{S_i},f>=\int f d\mu$. 
%\begin{lem}If  $\Phi\circ (cId+(1-c) \Psi)$  is an extreme points in the set of Markov operators then $\Phi \circ Id=\Phi \circ \Psi$ and so if $\Phi \circ(c Id+(1-c)\Psi)f=\int f$
 %for all $f$ then $ \Phi f =\int f$. 
%\end{lem}
\begin{lem}Let $\{D_i\}_{i=1}^\infty$ be a sequence of linear operators on $L^2$ with uniformly bounded operator norm. 
There exists a subsequence $\{n_i\}_{i=1}^\infty$ so that $\{D_{n_i}\}_{i=1}^\infty$ converges in the weak operator topology.
\end{lem}
This is a standard and straightforward fact. 
\begin{lem}Let $\{t_i\}_{i=1}^\infty$ be a $c$-partial rigidity sequence for $F_1^{s}$. If $\Psi$ is a weak operator limit point of $U_{F_1^{t_i}}$ then there exists a  operator $\Psi'$ with norm at most 1 so that $\Psi=c'Id+(1-c')\Psi'$ where $c'\geq c$. 
\end{lem}
Note that $\Psi'$ is a Markov operator from $L^2(\mu)$ to $L^2(\mu)$. So the composition of it with a Markov operator from $L^2(\nu)$ to $L^2(\mu)$ gives a Markov operator from $L^2(\nu)$ to $L^2(\mu)$. 
\begin{proof}
By condition (3) of Definition \ref{def:part rig}, for any pair of $L^2$ functions $f,g$ we have 
$$\underset{i \to \infty}{\lim}<\chi_{S_i}f,g> -\mu(S_i)<f,g>=0.$$ Applying conditions (2) and (3) we see  $\underset{i \to \infty}{\lim}<U_{F_1}^{t_i}(\chi_{S_i}f),g> -\mu(S_i)<f,g>=0$. With condition (1) this gives us the lemma. Indeed assume (after possibly passing to a subsequence) that $\underset{i \to \infty}{\lim}\mu(S_i)=c'$ and let $\Psi'$ be a weak operator limit  of the sequence of operators given by  $f \to \frac 1 {\mu(S_i^c)}U_{F_1}^{t_i}(\chi_{S_i^c}f)$. 
\end{proof}
%\begin{lem}$(F_1^s,\mu)$ and $(F_2^s,\nu)$ be two flows with eigenvalues for non constant functions $\{exp(s\alpha_i)\}$ and $\{exp(s\beta_i)\}$ respectively. If $\alpha_i \neq \beta_j$ for all $i,j$ then 
%\end{lem}
\begin{defin}\label{def:spread}Let $U^s$ be a strongly continuous one parameter unitary group. Let $\{exp(s\alpha_i)\}$ be the eigenvalues of $U^s$ that do not correspond to constant functions. We say $U^s$ satisfies condition (*) along a sequence $\{t_i\}_{i=1}^\infty$ if 
\begin{enumerate}
\item for all $j$ we have $\underset{i}{\inf} \|exp(t_i\alpha_j)-1\|>0$ and
\item $\{ U^{t_i}f\}$ converges weakly to 0 for all $f \in \mathcal{H}_c$.  
\end{enumerate} 
\end{defin}
\begin{prop}\label{prop:disjoint criterion} Let $(X, \mathcal{B}, F_1^s,\mu)$ and $(Y,\mathcal{B}',F_2^s,\nu)$ be two ergodic flows. Assume that $\{t_i\}_{i=1}^\infty$ is a $c$-partial rigidity sequence for  $U_{F_1}^s$ and ($\{t_i\}$, $U_{F_2}^s$) satisfies (*) then $(F_1^s,\mu)$ and $(F_2^s,\nu)$ are disjoint. 
\end{prop}
%Before proving the proposition we briefly recall some standard material (see \cite[page 132]{glasner}). A continuous linear map $\Phi:L^2((X,\mu)) \to L^2((Y,\nu))$ is called a \emph{Markov operator} if \begin{enumerate}
%\item $P\geq 0$ and $P^*\geq 0$,
%\item $P\chi_X=\chi_Y$ and $P^*\chi_Y=\chi_X$
%\item $PU_{F_1}^s=U_{F_2}^sP$. 
%\end{enumerate}{\sc check}
%Let $\lambda$ be an $F_1 \times F_2$ preserved measure of $X \times Y$ with marginals $\mu$ and $\nu$. This defines a  Markov operator $\Phi:L^2(\nu) \to L^2(\mu)$  by $\int_B\Phi(\chi_A)d\mu=\lambda(A \times B)$.   The set of Markov operators are in 1-1 correspondence with joinings. This identification respects the convex structure of preserved measures and so  the extreme points come from ergodic joinings. %$\Phi$ intertwines the dynamics: 
%%$\Phi_1 \circ U_{F_1}^s=U_{F_2}^s \circ \Phi_1$ and similarly for $\Phi_2$.
  
  Let $\lambda$ be an $F_1 \times F_2$ preserved measure of $X \times Y$ with marginals $\mu$ and $\nu$. We recall that $\lambda$ defines a  Markov operator $\Phi:L^2(\nu) \to L^2(\mu)$  by $\int_B\Phi(\chi_A)d\mu=\lambda(B \times A)$.  We prove $\lambda$ is the product measure by showing that $\Phi f=\int f $, so its Markov operator is the same as
 for  the product joining. 

 Let $t_{n_i}$ be a sequence so that $U_{F_2}^{t_{n_i}}$ has  a weak operator topology limit,  $\Gamma$, and $U_{F_1}^{t_{n_i}}$ has  a weak operator topology limit $cId+(1-c)\Psi.$ 
We now prove the proposition in three steps:
\begin{proof}
\noindent
\textit{Step 1:} For any $f \in \calH_{pp}(U_{F_2})$ with $\int f=0$ we have that $\Phi \circ \Gamma f=0$. 

It suffices to prove this for eigenfunctions of $F_2$. Let $f$ be an eigenfunction of $F_2^s$ with eigenvalue $exp(s\alpha)$. So $\Phi \circ U_{F_2}^sf=exp(s\alpha)\Phi f$ and $\Phi f$ is either an eigenfunction of $U_{F_1^s}$ with eigenvalue $exp(s\alpha)$ or $\Phi f=0$. Because $t_i$ is a partial rigidity sequence of $F_1$, $F_1$ does not have an eigenvalue $exp(\alpha)$ with $\underset{i}{\inf} |exp(t_i\alpha)-1|>0$ and so $\Phi f=0$. 

\noindent
\textit{Step 2:} For any $f \in \calH_{c}(U_{F_2})$ we have that $\Phi \circ \Gamma f=0$. 

By the second condition of (*) we have that $\Gamma f=0$, implying step 2. 

Step 1 and 2 imply that $\Phi \circ \Gamma f=((cId+(1-c)\Psi)\circ \Phi)f=\int f$. 

\noindent
\textit{Step 3:} $F_1^s\times F_2^s$ is $\mu \times \nu$ ergodic.

It is a standard result in ergodic theory that the product of two ergodic flows is ergodic iff they have disjoint pure point spectrum (modulo constants). Briefly an invariant function is an eigenfunction with eigenvalue 1. All eigenfunctions of a $F_1\times F_2$ have the form $(f(x),g(y))$ where $f,g$ are eigenfunctions of $F_1$ and $F_2$ respectively. The eigenvalue of this function is the product of the corresponding eigenvalues (see Lemma \ref{lem:product spectrum}). By our assumption this is 1 iff $f$ and $g$ are both constant (almost everywhere). 

By step three we have that the Markov operator $\Phi\circ \Gamma f=\color{black}((cId+(1-c)\Psi)\circ \Phi)f=\int f$ is an extreme point in the set of Markov operators and so $Id \circ \Phi=\Phi=\int f$. Thus $\lambda=\mu \times \nu$. 
\end{proof}

\section{Proof of Theorem \ref{thm:typ disjoint}} \label{sect:application}

The proof of Theorem \ref{thm:typ disjoint} is based on using the  $\bar{N}AN$ decomposition of $r_\theta \hat{h}_s$ to show there exists a set satisfying (*) for $F_{r_{\theta_1}\hat{h}_s\omega} \times...\times F_{r_{\theta_{n}}\hat{h}_s\omega}$ that contains a partial rigidity sequence for $F_{r_{\theta_{n+1}}\hat{h}_s\omega}$.  So now we consider matrix decompositions. 
\subsection{Matrix decomposition}\label{sec:matrix decomps}

First observe that if $\theta \neq \pm \frac {\pi}2$ we have 
$$r_\theta=\begin{pmatrix} 
\cos(\theta) & \sin(\theta)\\
-\sin(\theta) &\cos(\theta)
\end{pmatrix}=\begin{pmatrix} 
1 & 0\\
-\tan(\theta) &1
\end{pmatrix}\begin{pmatrix} 
\cos(\theta) & 0\\
0 &\sec(\theta)
\end{pmatrix}\begin{pmatrix} 
1 & \tan(\theta)\\
0 &1
\end{pmatrix}.$$

Next observe that 
$\begin{pmatrix} 
1& x\\
0 &1
\end{pmatrix}\begin{pmatrix} 
1& 0\\
s &1
\end{pmatrix}= \begin{pmatrix} 
1& 0\\
\frac{s}{1+xs} &1
\end{pmatrix}
\begin{pmatrix} 
1+xs& 0\\
0 &\frac 1 {1+xs}
\end{pmatrix}
\begin{pmatrix} 
1& \frac{x}{1+xs}\\
0 &1
\end{pmatrix}$
%\subsection{What this tells us}

We want to study the $\bar{N}AN$ decomposition of $r_\theta\hat{h}_s$. That is, writing $r_{\theta}\hat{h}_s$ as $\hat{h}_ag_bh_c$ for some $a,b,c$. In particular we want to observe the $AN$ coordinates and examine how the value of an $A$ coordinate changes with $s$ if we force the $N$ coordinate to be fixed.

By above the $N$ coordinate is $\begin{pmatrix} 
1& \frac{\tan(\theta)}{1+s\tan(\theta)}\\
0 &1
\end{pmatrix}.$ So if $c=\frac{\tan(\theta)}{1+s\tan(\theta)}$ then $\tan(\theta)=\frac{c}{1-cs}$ and $\theta=\arctan(\frac{c}{1-cs})$.

In turn the $A$ coordinate is 
$$\begin{pmatrix} 
\cos(\arctan(\frac{c}{1-cs}))& 0\\
0 &\sec(\arctan(\frac{c}{1-cs}))
\end{pmatrix}\begin{pmatrix} 
1+\frac{cs}{1-cs}& 0\\
0 &(1+\frac{cs}{1-cs})^{-1} 
\end{pmatrix}.$$
Now $\cos(\arctan(x))=\pm \frac 1 {\sqrt{1+x^2}}$ and so the $A$ coordinate is 
\begin{multline}\pm\begin{pmatrix}(1+(\frac{c}{1-cs})^2)^{-\frac 1 2} (1+\frac{cs}{1-cs})& 0\\
 0& \sqrt{1+(\frac c {1-cs})^2}(1+\frac{cs}{1-cs})^{-1} \end{pmatrix}\\ = \pm \begin{pmatrix}\frac 1 {\sqrt{(1-cs)^2+c^2}}&0\\0& \sqrt{(1-cs)^2+c^2} \end{pmatrix}.
  \end{multline}
  We will restrict our attention to $\theta \in (-\frac \pi 2 ,\frac \pi 2)$ where we obtain $ \begin{pmatrix}\frac 1 {\sqrt{(1-cs)^2+c^2}}&0\\0& \sqrt{(1-cs)^2+c^2} \end{pmatrix}$. If $\theta \in (\frac {\pi }2,\frac{3\pi}2)$ then we would obtain $- \begin{pmatrix}\frac 1 {\sqrt{(1-cs)^2+c^2}}&0\\0& \sqrt{(1-cs)^2+c^2} \end{pmatrix}$.
  Later in this section we will need the following straightforward result:
  \begin{lem}\label{lem:lin indep} Let $\eta_c$ be a function defined by $\eta_c(s)=\frac 1 {\sqrt{(1-cs)^2+c^2}}$. For almost every $c_1,...,c_n,c_{n+1}$ we have $\{(\frac{\eta_{c_i}}{\eta_{c_{n+1}}})'\}_{i=1}^n$ and 
  $\{(\frac{\eta_{c_{n+1}}}{\eta_{c_i}})'\}_{i=1}^n$ span $n$-dimensional subspaces. This implies for that for almost every $c_1,\dots, c_{n+1}$ we have that for any $(a_1,\dots, a_n)\neq \bar{0}$ the function $\sum_{i=1}^n a_i (\frac{\eta_{c_i}(s)}{\eta_{c_{n+1}}(s)})'$ and  $\sum_{i=1}^n a_i (\frac{\eta_{c_{n+1}}(s)}{\eta_{c_{i}}(s)})'$ take the value 0 on a finite subset of a compact interval of $\mathbb{R}$. 
  \end{lem}
  \begin{proof} We show the functions are different by showing they have different (non-removable) singularities. Note that since our functions ($\{(\frac{\eta_{c_i}}{\eta_{c_{n+1}}})'\}_{i=1}^n$ or 
  $\{(\frac{\eta_{c_{n+1}}}{\eta_{c_i}})'\}_{i=1}^n$) are holomorphic in a neighborhood of $\mathbb{R}$, if a linear combination is nonzero it implies that it takes the value 0 only finitely many times in a compact subset of $\mathbb{R}$. For convenience, let $Y$ be a two fold cover of the complex plane minus a finite number of points where our various $\eta$'s are defined and $\pi$ be the covering map.  Observe that $|\underset{s \to \pi^{-1}(\frac 1 c+i)}{\lim} \eta_c'(s)|=\infty$
  \footnote{To be explicit, $\underset{s \to \pi^{-1}p}{\lim}|\eta_y(s)|=\infty$ means that for any sequence $z_1,\dots\in Y$ so that $\pi (z_i) \to p$ we have $|\eta_y(z_i)| \to \infty.$} and for almost every $x$ we have that $\eta_x(s)'$ is bounded in a neighborhood of $\pi^{-1}(\frac 1 c+i) $.
     It follows that for almost every $x$ we have that $\underset{s \to \pi^{-1}(\frac 1 c +i)}{\lim}|\frac{\eta_c(s)'}{\eta_x(s)'})|=\infty$.\footnote{Observe that $\eta_c(s)'=\frac{-c(1-cs)}{((1-cs)^2+c^2)^{\frac 3 2 }}$.} For almost every $a,b$ we have that 
   $(\frac{\eta_a(s)}{\eta_b(s)})'$ is bounded in a neighborhood of $\pi ^{-1}(\frac 1 c+i)$. 
   Therefore for every $c$ we have a full measure set of $b_1,\dots, b_{n-1},x$ so that for any $a_1,\dots,a_{n-1}$ we have that $\sum a_i (\frac{\eta_{b_i}(s)}{\eta_x(s)})'$ is bounded in a neighborhood of $\pi^{-1}(\frac 1 c +i)$ while $(\frac{\eta_c(s)}{\eta_x(s)})'$ is not. This establishes that for almost every  $c_1,...,c_n,c_{n+1}$
    we have that $(\frac{\eta_{c_i}}{\eta_{c_{n+1}}})'$ is not in the span of $\{(\frac{\eta_{c_j}}{\eta_{c_{n+1}}})'\}_{j \neq i}$.   For the other set of functions we use that for almost every $x$ we have that $(\frac{\eta_x(s)}{\eta_c(s)})'=\frac{\eta_x(s)'\eta_c(s)-\eta_c(s)'\eta_x(s)}{\eta_c(s)^2}$ also has singularities at $\pi^{-1}(\frac 1 c +i)$.\footnote{For almost every $x$ we have $\frac{\eta_x'(s)}{\eta_c(s)}$ is zero and $\frac{\eta_x(s)\eta_c(s)'}{\eta_c(s)^2}=\eta_x(s) \frac{c(1-cs)((1-cs)^2+c^2)}{((1-cs)^+c^2)^{\frac 3 2 }}$ has singularities at these points.} 
  \end{proof}

\subsection{$\bar{N}$ and $A$ equivariance}
\begin{lem}\label{lem:NAN spectrum} Let the $\bar{N}AN$ decomposition of $M \in SL(2, \, \mathbb{R})$ be $\hat{h}_{s}g_\ell h_b$. Let $\{exp(t\alpha_i)\}_{i=1}^\infty$ be eigenvalues for $F_{h_b\omega}^t$ and $A$ be a set so that $U_{F_{h_b\omega}}^{t}$ converges to 0 in the weak operator topology on the subset of $L^2$ with continuous spectrum. Then $F^t_{M\omega}$ is isomorphic to $F^{e^{\ell}t}_{h_b\omega}$, the eigenvalues of $F^t_{M\omega}$ are $\{exp(te^{\ell}\alpha_i)\}_{i=1}^\infty$ and $e^{-\ell}A$ is a set so that $U_{F_{M\omega}}$ converges to 0 in the weak operator topology on the subset of $L^2$ with continuous spectrum. Also if $\{t_i\}_{i=1}^\infty$ is a $c$-partial rigidity sequence for $F_{h_b\omega}$ then $\{e^{-\ell} t_i\}_{i=1}^\infty$ is a $c$-partial rigidity sequence for $F_{M\omega}$.
\end{lem}
\begin{proof}The proof follows from two straightforward observations about how matrices act on vertical lines.  $\hat{h}_s$ is a measure isomorphism (because it acts isometrically on vertical leaves). Because it preserves vertical leaves as a set and contracts them by $e^{-t}$,  $g_t$ maps $\omega$ to $g_t\omega$ in such a way that $g_tF_{\omega}^{r}(p)=F_{g_t\omega}^{e^{-t}r}p$.  
\end{proof}

\subsection{Completion of the proof}
We need the  following standard lemma:

\begin{lem}\label{lem:deriv to measure2} Let $f$ be a differentiable, positive function on $[-k,k]$ with $f'$ bounded away from 0 and infinity. If $B,B'\subset f([-k,k])$ then 
$$\frac{\lambda(B)}{\lambda(B')}\frac{\min \, f'}{\max\, f'}\leq\frac{\lambda(f^{-1}B)}{\lambda(f^{-1}B')}\leq \frac{\lambda(B)}{\lambda(B')}\frac{\max \, f'}{\min \, f'}.$$
\end{lem}

\begin{lem}\label{lem:sing spectrum}Let $exp(\alpha_1),...,exp(\alpha_k)$ be in the unit circle in $\mathbb{C}$ and all be different from 1. Let $\{t_i\}_{i=1}^{\infty}$ be a sequence going to infinity. Let $f_1,...,f_k:[a,b] \to \mathbb{R}$ 
have that for any $(d_1,..,d_k)\neq (0,...,0)$ the function $\sum_{j=1}^kd_jf'_j$ is continuous and takes the value zero on a finite set. Then for any $\epsilon>0$ there exists $\delta>0$ so that 
$$ \underset{i \to \infty}{\liminf}\lambda(\{s: |exp(f_1(s)t_i\alpha_{1})\cdot exp(f_2(s)t_i\alpha_{2})...exp(f_k(s)t_i\alpha_k)-1|>\delta\})>b-a-\epsilon.$$
\end{lem}
\begin{proof}

There exists $C$ so that for any $u$,
$$|exp(f_1(s)t_i\alpha_{1})\cdot exp(f_2(s)t_i\alpha_{2})...exp(f_k(s)t_i\alpha_k)-1|>u$$ whenever
$$d(\sum f_\ell(s) \alpha_{\ell}t_i,2\pi\mathbb{Z})>Cu.$$

By our assumption on the $f_j$ there exists, $D>0$, $U \subset [a,b]$ a finite union of intervals, with
 $\lambda(U)>b-a-\frac \epsilon 2$, so that $\frac 1 D<|\sum f'_\ell(s) \alpha_{\ell}|<D$ for all $s \in U$.  
We choose $0<\delta<\frac{\epsilon}{8D^2}$. Let $I_1,...,I_n$ be the disjoint intervals that make up $U$. 
There exist $e,e'$ so that  $\sum f_\ell(I_a) \alpha_{\ell}t_i=[t_ie,t_ie']$ for all $i$. 
This interval has length $t_i|e'-e|$ and $\cup_{n \in \mathbb{Z}}B(n,\delta)\cap [t_ie,t_ie']$ has measure at most $\lceil t_i|e-e'|\rceil 2\delta$. 
Because $t_i$ goes to infinity there exists $i_0$ so that for $i>i_0$  we have
 $$\lambda( \cup_{n \in \mathbb{Z}}B(n,\delta)\cap [t_ie,t_ie'])<3t_i|e-e'|\delta\leq\frac 3 {8D^2}\epsilon t_i|e-e'|.$$ 
Its complement in $[t_ie,t_ie']$ has measure at least 
$(1-\frac 3{8D^2}\epsilon )t_i|e-e'|$.  
Since the ratio of the maximum and minimum of the derivative of $\sum f_\ell(s) \alpha_{\ell}t_i$ is at most $D^2$, 
by Lemma \ref{lem:deriv to measure2} we have 
$\lambda(\{s \in I_a: d(\sum f_\ell(s) \alpha_{\ell}t_i,2\pi\mathbb{Z})<\delta\})<\frac\epsilon 2|I_a|$, if $\epsilon$ is small enough.
%$\lambda(\{s \in I_a: d(\sum f_\ell(I_a) \theta_{j_\ell}^{(\ell)}t_i,\mathbb{Z})>\delta\})<\frac\epsilon 2|I_a|$, if $\epsilon$ is small enough.\footnote{We want $\frac{\epsilon \frac 3 {8D^2}}{1-\epsilon \frac 3 {8D^2}}<\frac \epsilon {2D^2}.$}  
The lemma follows by applying this argument to each of the $I_a$.
\end{proof}

%%%%%%%%%%%%%%%%%%%%%%%%%
%%%%%%%%%%%%%%%%%%%%%%%%%%
\begin{rem} The next corollary is technical to state but it says that if $\{exp(\alpha_i^{(\ell)})\}_{i=1}^\infty$ are eigenvalues for a flow $F_\ell$ then for many $s$, $t_1,..$ contains an unbounded subsequence satisfying condition 1 of (*) for $F_1^{f_1(s)} \times... \times F_{k}^{f_{k}(s)}$.
\end{rem}
\begin{cor}\label{cor:sing spectrum}Let  $exp(\alpha_1^{(1)}),exp(\alpha_2^{(1)}),... ,exp(\alpha^{(2)}_1),....,exp(\alpha^{(k)}_1), \dots$  be in the unit circle in $\mathbb{C}$ and never 1 and $P$ be a finite set. Let $f_1,f_2,...,f_k %\frac{f_i}{f_j}, \frac{f_if_j}{f_k},...\frac{f_{i_1}...f_{i_{k-1}}}{f_{i_k}}
  :[a,b] \setminus P \to \mathbb{R}$ be continuously differentiable, have derivatives that take the value 0 on a finite set and so that for all $\delta>0$ we have that their derivatives are bounded on $\{x\in [a,b]:d(x,P)>\delta\}$. Also assume $f_1,...,f_k$ has that for any $(d_1,..,d_k)\neq \bar{0}$, the function $\sum_{j=1}^kd_jf'_j$ is continuous and takes the value 0 on a finite set. Let $\{t_i\}_{i=1}^\infty$ be unbounded. Then for every $\epsilon>0$ there exists $\delta_{\bar{i}}>0$ for each $\bar{i}\in \mathbb{N}^k$
so that for all $n$ we have 
\begin{multline}
\underset{i \to \infty}{\liminf}\,
\lambda(\{s \in [a,b]: \\ 
 |exp(f_1(s)t_i\alpha_{j_1}^{(1)}) \cdot exp(f_2(s)t_i\alpha_{j_2}^{(2)})...exp(f_k(s)t_i\alpha_{j_k}^{(k)})-1|>\delta_{j_1,...,j_k}\\
\text{ for all $(j_1,...,j_k)\in \{1,...,n\}^k$}\})\geq b-a-\epsilon.
\end{multline}
%\marginpar{Use vector notation for the $\delta_{i_1,...,i_k}$}
\end{cor}
It suffices to prove this on a connected component of $[a,b]\setminus P$. 
Choose $\{\epsilon_{\bar{j}}\}_{\bar{j}\in \mathbb{N}^k}$ so that each one is positive and the sum of them all is at most $\epsilon$. For each such $\epsilon_{j_1,...,j_k}$ choose $\delta_{j_1,...,j_k}$ by the previous lemma. The corollary follows.

%\marginpar{MAYBE DELETE THIS AND THE COROLLARY}\begin{defin}Given two flow $F_1,F_2$ we say that $F_2$ is \emph{eigen-free} to $F_1$ if there exists a strong partial rigidity sequence for $F_2$, $t_1,...$,  so that $\underset{j \to \infty}{\liminf} \, |\alpha^{t_j}-1|>0$ for every $\alpha$ an eigenvalue of $F_1$  other than 1.\marginpar{This is a horrible name} %not corresponding to a constant function
%\end{defin}
%\begin{cor} Let $F_1,..,F_k, F_{k+1}$ be flows and $F_{k+1}$ have a $c$-strong partial rigidity sequence.   Let $f_1,....,f_k:[a,b] \to \mathbb{R}$ 
%have that for any $c_1,..,c_k$ the function $\sum_{j=1}^kc_jf'_j$ is continuous and takes the value zero on an isolated set. %Let $t_1,..$ be an unbounded strong partial rigidity sequence for $F_2$. 
%Then
%   $$\lambda(\{s \in [a,b]: F_{k+1}^{t} \text{ is eigen-free to }F_1^{f_1(s)t}\times...\times F_{k}^{f_k(s)t}\})>(b-a)-\epsilon.$$
%\end{cor}
%\begin{defin} Let $A \subset \mathbb{R}$. The \emph{density of $A$} is $\underset{N \to \infty}{\liminf}\, \frac{\lambda(A \cap [-N,N])}{2N}$. 
%\end{defin}
\begin{lem} Let $A \subset \mathbb{R}$ have density 1, $f:[a,b]\to \mathbb{R}$ and $f'$ be continuous and takes the value zero on a  finite set. Let $\{t_i\}_{i=1}^{\infty}$ be a sequence going to infinity. Then $$\underset{i \to \infty}{\liminf}\, \lambda(\{s\in [a,b]:t_i \in f(s)A\})=b-a.$$
\end{lem}
This is similar to the proof of Lemma \ref{lem:sing spectrum}.
\begin{proof}It suffices to show that for all $\epsilon>0$ and small enough there exist $i_0$ so that $ \lambda(\{s\in [a,b]:t_i \in f(s)A\})\geq b-a-\epsilon$ for all $i>i_0$. 
Let $\epsilon>0$ be given and be small enough. By our assumption on $f$ there exists $D>0$, $U$, a finite union of intervals,  so that $\frac 1 D<|f'(s)|<D$ for all $ s\in U$ and $\lambda(U)>b-a-\frac \epsilon 2$. Let $I_1,...,I_n$ be the disjoint intervals that make up $U$. For each $I_a$ there exists $e,e'$ so that $f(I_a)t_i=[t_ie,t_ie']$ for all $t_i$. 
Choose $i_a$ so that  for $i\geq i_a$ we have $\lambda([t_ie,t_ie']\cap A)>(1-\frac \epsilon {4D^2}) t_i|e-e'|$. As in the proof of Lemma \ref{lem:sing spectrum}, if $\epsilon>0$ is small enough we have 
$\lambda(\{s\in [a,b]: f(s)t_i\in A\})>b-a-\frac \epsilon 2$.
 Let $i_0=\max \{i_a\}_{a=1}^k$ and the lemma follows.
\end{proof}
\begin{cor}\label{cor:cont spectrum}  Let $A_1,...,A_{k} \subset \mathbb{R}$ have density 1, $P$ be a finite set, $f_1,..,f_{k}:[a,b]\setminus P \to \mathbb{R}$ have $f_1',...,f'_{k}$ are continuous and take the value zero on a finite set. Let $\{t_i\}_{i=1}^{\infty}$ be a sequence going to infinity. Then $$\underset{i \to \infty}{\liminf}\, \lambda(\{s\in [a,b]:t_i \in f_1(s)A_1\cap... \cap f_{k}(s)A_{k}\})=b-a.$$
\end{cor}

%\begin{cor}\label{cor:cont spectrum}  Let $A_1,...,A_n \subset \mathbb{R}$ have density 1, $f_1,..,f_n:[a,b]\to \mathbb{R}$ have $f_1',...,f'_n$ are continuous and take the value zero on an isolated set. Let $t_1,..$ be a sequence going to infinity. Then $\underset{i \to \infty}{\liminf}\, \lambda(\{s\in [a,b]:t_i \in f_1(s)A_1\cap... \cap f_n(s)(A_n)\})=b-a.$
%\end{cor}
Let $\Phi_s:\mathbb{R}\to (-\frac \pi 2,\frac \pi 2)$ by $\Phi_s(a)=\theta$ where the $\bar{N}AN$ decomposition of $r_\theta\hat{h}_s=\hat{h}_bg_th_a$ for some $b,t$. 
Note by Section \ref{sec:matrix decomps} $\Phi_s$ is well defined. Indeed we showed that the $N$ coordinate in this decomposition of $r_\theta\hat{h}_s$ is $\frac{\tan(\theta)}{1+s\tan(\theta)}$ and $\tan:(-\frac \pi 2,\frac \pi 2) \to \mathbb{R}$ is a bijection. Indeed $\tan$ is injective on $(-\frac \pi 2,\frac \pi 2)$ and if $\begin{pmatrix}1&s\\0&1\end{pmatrix}\begin{pmatrix}1&0\\t&1\end{pmatrix}$ has the same $N$ coordinate as $\begin{pmatrix}1&s'\\0&1\end{pmatrix}\begin{pmatrix}1&0\\t&1\end{pmatrix}$ then $s=s'$.

The next proposition is a main step in the proof. It needs some notation. If constants $c_1,...,c_k,c_{k+1}$ are understood let $\zeta_i(s)= \frac 1 {\sqrt{(1-c_is)^2+c_i^2}}.$ From Section \ref{sec:matrix decomps} and Lemma \ref{lem:NAN spectrum} we have $F^{t}_{r_{\Phi_s(c_i)}\hat{h}_s\omega}$ is isomorphic to $F^{\zeta_i(s)t}_{h_{c_i}\omega}.$ Indeed, from Section \ref{sec:matrix decomps} if  $r_\theta \hat{h}_s=\hat{h}_bg_th_{c_i}$ then $t=\log( \frac 1 {\sqrt{(1-c_is)^2+c_i^2}})$.

\begin{prop}\label{prop:horocycle disjoint}  If $\lambda^k$ almost every $(c_1,...,c_k)$ we have that for almost every $s$ the flow 
 $F_{r_{\Phi_s(c_1)}\hat{h}_s\omega} \times ...\times F_{r_{\Phi_s(c_k)}\hat{h}_s\omega}$ is $\lambda_{\hat{h}_s\omega}^k$ ergodic then for  
 $\lambda^{k+1}$ almost every $(c_1,...,c_k,c_{k+1})$ we have that for almost every $s$ the flow 
 $F_{r_{\Phi_s(c_1)}\hat{h}_s\omega} \times ...\times F_{r_{\Phi_s(c_k)}\hat{h}_s\omega}$ is disjoint from $F_{r_{\Phi_s(c_{k+1})}\hat{h}_s\omega}$ as systems preserving $\lambda_{\hat{h}_s\omega}^k$ and $\lambda_{\hat{h}_s\omega}$ respectively.
\end{prop}
The proof uses the following result which is known to experts and for that reason the proof is in Appendix A.
\begin{thm}\label{thm:part rig flow}For every  translation surface $\omega$, there exists  $a>0$ so that for almost every $\theta$ we have $F_{r_\theta\omega}$ has a $\frac a 2$-partial rigidity sequence.
\end{thm}

\begin{proof}[Proof of Proposition \ref{prop:horocycle disjoint}] We restrict to $s$ in a fixed bounded interval. It suffices to show that for  all but a set of $s$ of measure $\epsilon$ we have that 
 $$F_{h_{c_1}\omega }^{\zeta_1(s)t}\times...\times F_{h_{c_k}\omega }^{\zeta_k(s)t}$$
 is ergodic and has a set satisfying (*) that contains an $c>0$-partial rigidity sequence for 
  $F^{\zeta_{k+1}(s)t}_{h_{c_{k+1}}\omega }$.  By the assumption of the proposition we may assume the flow is ergodic. 
 If $\{t_i\}_{i=1}^\infty$ is a $c$-partial rigidity sequence for $F^t_{h_{c_{k+1}}\omega }$ (which exists by Theorem \ref{thm:part rig flow}), $A_i$ are sets
  along which $U_{F_{h_{c_i}\omega }}$ converges to zero  on $\mathcal{H}_c$  in the weak operator topology 
  and $\{\alpha_1^{(i)}\}_{i=1}^\infty$ are the non-constant eigenvalues of $F^t_{h_{c_i}\omega }$ then it suffices to show that there exist $\{\delta_{\bar{j}}>0\}_{\bar{j}\in \mathbb{N}^k}$ so that for any $n \in \mathbb{N}$ for all but a set of $s$ of measure $\epsilon$ we have that there exist an infinite number of $t_i$ simultaneously satisfying:
 \begin{enumerate}
 \item $\frac{t_i}{\zeta_{k+1}(s)}\in \cap_{j=1}^k\frac 1 {\zeta_j(s)}A_i.$
 \item For each $j_1,...,j_k\in \{1,...,n\}^k$ we have 
$$
 |exp(\zeta_1(s)\frac{t_i}{\zeta_{k+1}(s)}\alpha_{j_1}^{(1)})...
 exp(\zeta_k(s)\frac{t_i}{\zeta_{k+1}(s)}\alpha_{j_k}^{(k)})-1|>\delta_{j_1,...,j_k}.
$$
 \end{enumerate}
 We now justify (1). Consider the functions 
  $$f_i(s)=\frac{\zeta_{k+1}(s)}{\zeta_i(s)}.$$
  Observe that $t \in f_i(s)A_i$ iff 
  $$\frac{t}{\zeta_{k+1}(s)}\in \frac 1 {\zeta_i(s)}A_i.$$  By Lemma \ref{lem:lin indep}, for almost every choice of $c_i$, the functions $f_i$ satisfy the assumptions of Corollary \ref{cor:cont spectrum} and so it satisfies  condition 1 for each large enough $t_i$ on all but a set of $s$ of measure $\epsilon$. 
  
  We now justify (2). Consider the functions 
  $$\hat{f}_i(s)=\frac 1 {f_i(s)}.$$
  Observe that  
  $$ |exp(\zeta_1(s) \frac{t_i}{\zeta_{k+1}(s)}\alpha_{j_1}^{(1)})...
 exp(\zeta_k(s)\frac{t_i}{\zeta_{k+1}(s)}\alpha_{j_k}^{(k)})-1|>\delta_{j_1,...,j_k}$$
 iff 
$$|exp(\hat{f}_1(s)t_i\alpha_{j_1}^{(1)})...exp(\hat{f}_k(s)t_i\alpha_{j_k}^{(k)}) -1|>\delta_{j_1,....,j_k}.$$ Since the functions $\hat{f}_i$ satisfy the assumptions of  Corollary \ref{cor:sing spectrum} for almost every $(c_1,..,c_k)$ (by Lemma \ref{lem:lin indep}), by Corollary \ref{cor:sing spectrum} we have that for any fixed $n$ condition 2 is satisfied for all large enough $t_i$ \footnote{Where the largeness of $t_i$ can depend on $n$} on all but a set of $s$ of measure $\epsilon$. Putting these together we have that for all but a set of $s$ of measure $2\epsilon$ for any $n$ we satisfy conditions 1 and 2 simultaneously for arbitrarily large $t_i$.
\end{proof}
%To prove Proposition \ref{prop:crit holds} we need to justify that we can apply Fubini's Theorem, and so we state the next well known lemma whose proof is straightforward:
To prove Theorem \ref{thm:typ disjoint} we will need the following two known and straightforward results:
\begin{lem}\label{lem:disjoint ue} If $(F_1,\lambda_1),(F_2,\lambda_2)$ are uniquely ergodic and disjoint then $F_1\times F_2$ is uniquely ergodic.
\end{lem}
\noindent 
\begin{lem}\label{lem:ue borel} Uniquely ergodic is Borel and therefore measurable.\footnote{It suffices to show that for a countable number of metric balls that generate the topology, the Birkhoff averages converge uniformly. For each such ball this is a Borel condition.}
\end{lem}
\begin{proof}[Proof of Theorem \ref{thm:typ disjoint}] We prove this by induction on $k$. 
The base case follows from Kerckhoff-Masur-Smillie \cite{KMS}. We now assume Theorem \ref{thm:typ disjoint} 
is true for $k$ and prove it for $k+1$ by the previous proposition and Fubini's Theorem. 
Since we are assuming that $F_{\Phi_s(c_1)\hat{h}_s\omega}\times...\times F_{\Phi_s(c_k)\hat{h}_s\omega}$ is ergodic, we may apply the previous proposition. Therefore, 
 for  $\lambda^{k+1}$ almost every $(c_1,...,c_{k},c_{k+1})$ we have that for almost every $s$, that $F_{\Phi_s(c_1)\hat{h}_s\omega}\times...\times F_{\Phi_s(c_k)\hat{h}_s\omega}$ is disjoint from $F_{\Phi _s(c_{k+1})\hat{h}_s\omega}$. 
  By Fubini's Theorem, which is justified by Lemma \ref{lem:ue borel}, by Proposition \ref{prop:horocycle disjoint}, we have that for almost every $s$ for almost every  $(c_1,...,c_k,c_{k+1})$ we have that $F_{\Phi_s(c_1)\hat{h}_s\omega}\times...\times F_{\Phi_s(c_k)\hat{h}_s\omega}$ is disjoint from $F_{\Phi _s(c_{k+1})\hat{h}_s\omega}$.  By Lemma \ref{lem:disjoint ue} we have $(F_{\Phi_s(c_1)\hat{h}_s\omega}\times...\times F_{\Phi_s(c_k)\hat{h}_s\omega}) \times F_{\Phi _s(c_{k+1})\hat{h}_s\omega}$ is uniquely ergodic. 
 Since $\Phi_s$ and its inverse are absolutely continuous we have the Theorem. \end{proof} 

%\begin{proof}[Proof of Theorem \ref{thm:crit holds}] It suffices to show that for almost every $(k+1)$-tuple $h_{c_1},...,h_{c_{k+1}}$
% for all but a set of $s$ of measure $\epsilon$ we have that 
% $$F^{\cos(\arctan(\frac{c_1}{1-c_1s}))+1+\frac{c_1}{1-c_1s}}\times...\times F^{\cos(\arctan(\frac{c_k}{1-c_ks}))+1+\frac{c_k}{1-c_ks}}$$
%  has a good sequence that contains a $e>0$-strong partial rigidity sequence for 
%  $F^{\cos(\arctan(\frac{c_{k+1}}{1-c_{k+1}s}))+1+\frac{c_{k+1}}{1-c_{k+1}s}}$.\footnote{Briefly, this is true for all $\ell\leq k$ then we have the condition for good for disjointness. This is an intersection of $k$ conditions and so it is true for $k\epsilon$ of our points} Consider the functions 
%  $$f_i(s)=\frac{\cos(\arctan(\frac{c_i}{1-c_is}))+1+\frac{c_i}{1-c_is}}
%  {\cos(\arctan(\frac{c_{k+1}}{1-c_{k+1}s}))+1+\frac{c_{k+1}}{1-c_{k+1}s}}$$ and observe that for almost every $(k+1)$-tuple they satisfy the %conditions of Lemma \ref{lem:sing spectrum} and Proposition \ref{prop:intersection}. By Proposition \ref{prop:intersection} Condition 1 of a good sequence is taken care of. By Lemma \ref{lem:sing spectrum} Condition 2 of a good sequence is taken care of. 
%\end{proof}

\appendix
\section{Proof of Theorem \ref{thm:part rig flow}}
%For completeness we prove that on every translation surface the flow in almost every direction has a strong partial rigidity sequence. This is known to experts.
This section proves Theorem \ref{thm:part rig flow} for completeness. This result  is known but we could not find in print.
\begin{prop}\label{prop:suff cyl}It suffices to show that for every surface $\omega$ there exists $c>0$ so that for every $\epsilon>0, N$, for almost every $\theta$ there exists a cylinder with area at least $c$, length $T>N$ and direction $\psi$ so that $|\psi-\theta|<\frac \epsilon {T^2}$. 
\end{prop}

\begin{lem}\label{lem:key} Suppose $\omega$ has a cylinder $C$ in direction $\phi$ with  period $T\geq 1$. Let $\theta$ be a direction such that $|\theta-\phi|<  \epsilon/T^2$. If $F_{\theta}^t(x) $ is inside $C$ for all $0<t<T $ then 
$d(F^t_{\theta}(x),F^t _{\phi}(x))<\epsilon/T$. In particular, $d(F^T_\theta (x), x) < \epsilon/T$. 
\end{lem}
\begin{proof}
Since a cylinder is convex, there is no singularity along the flat geodesic inside the cylinder $C$ between $F^t_{\theta}(x)$ and $F^t_{\phi}(x)$. Consequently, 

\[d(F^t_{\theta}(x),F^t_{\phi}(x))= 2t  \sin \left(\frac{\vert \theta-\phi \vert}{2}\right) \leqslant t \vert \theta-\phi \vert < \epsilon/T.\] 
\end{proof}
%\begin{cor}\label{cor:key} Suppose $\omega$ has a cylinder $C$ in direction $\phi$ with  period $T\geq 1$. Let $\theta$ be a direction such that $|\theta-\phi|<  \epsilon/T^2$. If $F_{\theta}^t(x) $ is inside $C$ for all $0<t<T $ then 
%$d(F^t_{\theta}(x),F^t _{\phi}(x))<\epsilon/T$. In particular, $d(F^{T}_\theta (x), x) < \epsilon/T$. 
%\end{cor}
\begin{proof}[Proof of Proposition \ref{prop:suff cyl}] 
Let $\delta>0$, and $N$ be given. By our assumption of the proposition there exists a cylinder of length $T>N$, area at least $c$ and direction within $\frac{\delta}{4T}$ of the vertical. At least $c-\delta$ of the points in the cylinder satisfy the assumption of Lemma \ref{lem:key}. Let $S_T$ be the set of these points. By Lemma \ref{lem:key} we have
 $\int_{S_T}|d(F^{T}(x),x)|d\lambda_\omega(x)<\frac{\delta}{T}<\delta.$ Choosing $N_i$ going to infinity and obtaining corresponding $T_i$, we obtain a $\frac c 2$ partial rigidity sequence $T_i$ with corresponding sets $S_{T_i}$.
\end{proof}
We next paraphrase a result \cite{MTW} which follows from that statement by choosing the identity function to be the dimension function and the fact that $|\sin(\theta)|<2|\theta|$ for all $\theta$ close enough to 0.
\begin{thm}(\cite[Theorem 6.1 (b)]{MTW}) Let $f:[0,\infty)\to (0,\infty)$ be non-increasing and have $\int_C^\infty r f(r)dr=\infty$ for all $C>0$. For every  translation surface there exists $a>0$ (depending only on genus) so that and almost every $\theta$ there exists infinitely many $L$ so that there is a cylinder of length $L$, area at least $c$ in direction $\psi$ where $|\psi-\theta|< 2\frac{f(L)}{L}$.
\end{thm}
\begin{cor}For every surface $\omega$ there exists $c>0$ so that for every $\epsilon>0, N$, for almost every $\theta$ there exists a cylinder with area at least $c$, length $L>N$ and direction $\psi$ so that $|\psi-\theta|<\frac \epsilon {L^2}$. 
\end{cor}

\section{Open questions}

\begin{ques}Is every surface weakly illuminated by circles? Illuminated by circles?
\end{ques}
\begin{ques}Is every surface where weak mixing is typical weakly illuminated by circles? 
\end{ques}
The obvious obstruction to applying our methods is that $(p,p)$ may not be generic for $(\lambda^2)^2$ under typical $F_{\theta}\times F_{\phi}$. 
To highlight the short comings of our methods we present the following questions:
\begin{ques}Is there a surface with a sequence $t_i$ tending to infinity
 so that $\underset{i \to \infty}{\lim} \, F_{\theta}^{t_i}(p)=p$ for almost every $\theta$? Does almost every surface in genus at least 2 have this property?
\end{ques}
\begin{ques}\label{ques:iso}Is there a surface where the isomorphism class of the flow in the vertical direction has positive (or full) measure?
\end{ques}
Of course, we suspect the answer to Question \ref{ques:iso} is no, but we do not know how to prove this.

%\begin{defin}If a surface has the property that  for any $k \in \mathbb{N}$ we have \marginpar{Delete?}
%$$\{(\theta_1,...,\theta_k):F_{\theta_1}\times ...\times F_{\theta_k} \text{ is uniquely ergodic} \}$$
%has full measure then we say it is \emph{good for distribution}.
%\end{defin}
%
%Some other question occurred in the process of working on this paper:
%\begin{ques}Does every flat surface have the property that the flow in almost every direction is rigid? Is this true for Veech surfaces? Is there an orbit closure other than a Teichm\"uller curve which does not contain single cylinder surfaces?
%\end{ques}
%


\begin{thebibliography}{xxx}

\bibitem{Disjoint} J.~Chaika, \textit{Every ergodic transformation is disjoint from almost every interval exchange transformation}, Ann. of Math. (2) 175 (2012), no. 1, 237--253.

\bibitem{Disjoint flow} J.~Chaika, V.~Gadre, \textit{Every transformation is disjoint from almost every non-classical exchange}, Geom. Dedicata 173 (2014), 105--127.

\bibitem{Eskin-Mirzakhani} A.~Eskin, M.~Mirzakhani, \textit{Invariant and stationary measures for the $\textrm{SL}(2,\R)$ action on moduli space}, preprint  2013.

%\bibitem{Keane} M.~Keane, \textit{Interval Exchange Trasformations}, Mathematische Zeitschrift, v. 141,  25-31, 1975.

\bibitem{Forni-Matheus} G. Forni, C. Matheus, \textit{Introduction to Teichm\"uller theory and its applications to dynamics of interval exchange transformations, flows on surfaces and billiards},  J. Mod. Dyn. 8 (2014),  3-4, 271--436.

\bibitem{glasner}E. Glasner. \textbf{Ergodic theory via joinings.} Mathematical Surveys and Monographs, 101. American Mathematical Society, Providence, RI, 2003. xii+384 pp.

\bibitem{KMS} S.~Kerckhoff, H.~Masur; J.~Smillie, \textit{Ergodicity of billiard flows and quadratic differentials}, Ann. of Math. (2) 124 (1986), no. 2, 293--311

\bibitem{LMW} S. Leli\`evre, T. Monteil, B. Weiss, \textit{Everything is illuminated}, Geom. Topol. 20 (2016), no. 3, 1737--1762.
 
 \bibitem{MTW} L.~Marchese, R.~Trevi\~no, S.~Weil, \textit{Diophantine approximations for translation surfaces and planar resonant sets,} arxiv:1502.05007.

%\bibitem{Masur1} H.~Masur, \textit{Interval exchange transformations and measured foliations}, Annals of Mathematics, v. 115, 169-200, 1982.

\bibitem{MT} H.~Masur, S.~Tabachnikov, \textit{Rational billiards and flat structures}, Handbook of dynamical systems, Vol. 1A, 1015--1089, North-Holland, Amsterdam, 2002.

\bibitem{Mon} T.~Monteil, \textit{ Illumination dans les billards polygonaux et dynamique symbolique}, Thesis, 2005.

\bibitem{Ru} D.~Rudolph, \textit{Fundamentals of measurable dynamics. Ergodic theory on Lebesgue spaces}, Oxford Science Publications. The Clarendon Press, Oxford University Press, New York, 1990. x+168 pp.

\bibitem{viana survey} M. Viana, \textit{ Ergodic theory of interval exchange maps.} Rev. Mat. Complut. 19 (2006), no. 1, 7--100.

\bibitem{Zo} A. Zorich, \textit{Flat surfaces}, Frontiers in number theory, physics, and geometry. I, 437--583, Springer, Berlin, 2006.

%\bibitem{Veech}W.~Veech, \textit{Gauss measures for transformations on the space of interval exchange maps}, Annals of Mathematics, v. 115, 201-242, 1982.

\end{thebibliography}
\end{document}